\documentclass[a4paper,10pt,normalheadings,makeidx]{article}
\usepackage{graphicx} 
\usepackage{amsmath,amssymb,amsthm} 
\usepackage[utf8]{inputenc}
\usepackage{hyperref}
\usepackage{lineno}
\usepackage{lipsum}
\usepackage{color}

\newpage
\newtheorem{Corollary}{Corollary}
\newtheorem{Theorem}{Theorem}
\newtheorem{Lemma}{Lemma}
\newtheorem{Claim}{Claim}

\newtheorem{Proposition}{Proposition}

 \title{On the complexity of finding well-balanced orientations with upper bounds on the out-degrees}
\author{Florian H\"orsch, Zolt\'an Szigeti}

\begin{document}
\maketitle

 \begin{abstract}
We show that the problem of deciding whether a given graph $G$ has a well-balanced orientation $\vec{G}$ such that $d_{\vec{G}}^+(v)\leq \ell(v)$ for all $v \in V(G)$ for a given function $\ell:V(G)\rightarrow \mathbb{Z}_{\geq 0}$ is NP-complete. We also prove a similar result for best-balanced orientations. This improves a result of Bern\' ath, Iwata, Kir\' aly, Kir\'aly and Szigeti and answers a question of Frank.
\end{abstract}

\section{Introduction}

This article contains a negative result concerning the possibility of deciding whether a given graph has a well-balanced or best-balanced orientation with a certain extra property. Any undefined notions can be found in Section \ref{prel}.
\medskip

During the history of graph orientations, the problem of characterizing graphs admitting orientations with certain connectivity properties has played a decisive role. The first important theorem due to Robbins \cite{r39} states that a graph has a strongly connected orientation if and only if it is 2-edge-connected. In 1960, Nash-Williams \cite{nw60} proved several theorems generalizing the result of Robbins. The first one is the following natural generalization of the result of Robbins to higher global arc-connectivity.

\begin{Theorem}\label{nwfaible}
Let $G$ be a graph and $k$ a positive integer. Then $G$ has a $k$-arc-connected orientation if and only if $G$ is $2k$-edge-connected.
\end{Theorem}

While Theorem \ref{nwfaible} resolves the problem of finding graph orientations of high global arc-connectivity, Nash-Williams also considered orientations satisfying local arc-connectivity conditions. We say that an orientation $\vec{G}$ of a graph $G$ is well-balanced if $\lambda_{\vec{G}}(u,v)\geq \lfloor \frac{\lambda_{G}(u,v)}{2}\rfloor $ for all $(u,v) \in V(G)\times V(G)$. If additionally $d_{\vec{G}}^+(v) \in \{\lfloor \frac{d_G(v)}{2}\rfloor,\lceil \frac{d_G(v)}{2}\rceil\}$ holds for all $v \in V(G)$, then $\vec{G}$ is called best-balanced. Nash-Williams proved the following result in \cite{nw60}.

\begin{Theorem}\label{nwfort}
Every graph has a best-balanced orientation.
\end{Theorem}

Observe that Theorem \ref{nwfort} implies Theorem \ref{nwfaible}. In the last decades, numerous attempts have been made to develop theory surrounding Theorems \ref{nwfaible} and \ref{nwfort}. These attempts turned out to be much more successful when concerning Theorem \ref{nwfaible} than when concerning Theorem \ref{nwfort}. For example, while a relatively simple proof of Theorem \ref{nwfaible} relying on a splitting off theorem of Lov\' asz has been found by Frank \cite{book}, no simple proof of Theorem \ref{nwfort} is known. Even though since the original, very complicated proof of Nash-Williams new proofs have been found by Mader \cite{m78} and Frank \cite{f127}, all of them are pretty involved. 

Another branch of research in the theory sorrounding Theorems \ref{nwfaible} and \ref{nwfort} consists in characterizing graphs which admit orientations satisfying some extra properties in addition to the connectivity conditions. These problems turn out to be much more tractable when trying to generalize Theorem \ref{nwfaible} than when trying to generalize Theorem \ref{nwfort}.

 For generalizing Theorem \ref{nwfaible}, polymatroid theory has proven to be a valuable tool. It allowed Frank \cite{book} to solve the problem of deciding whether a mixed graph has a $k$-arc-connected orientation for some given positive integer $k$ and to solve the more general problem of finding a minimum cost $k$-arc-connected orientation of a given graph where a cost is given for both possible orientations of each edge.

In \cite{bikks}, Bern\'ath et al. attempted to obtain similar generalizations for Theorem \ref{nwfort} which yielded several negative results, see also \cite{b}. For example, the problems of finding well-balanced and best-balanced orientations minimizing a given weight function were proven to be NP-complete in \cite{bikks}. The problem of deciding whether a mixed graph has a best-balanced orientation has also been proven to be NP-complete in \cite{bikks}. A proof that the problem of deciding whether a mixed graph has a well-balanced orientation is NP-complete has been found by Bern\' ath and Joret \cite{bj}.

Another extra property which can be imposed on the orientation is degree constraints. Here a generalization of Theorem \ref{nwfaible} has been obtained by Frank \cite{frank} using comparatively elementary methods. As its proof is constructive, we obtain the following result.

\begin{Theorem}
There is a polynomial time algorithm which, given a graph $G$, a positive integer $k$ and two functions $\ell_1,\ell_2:V(G) \rightarrow \mathbb{Z}_{\geq 0}$, decides whether there is a $k$-arc-connected orientation $\vec{G}$ of $G$ such that $\ell_1(v)\leq d_{\vec{G}}^+(v)\leq \ell_2(v)$ for all $v \in V(G)$.
\end{Theorem} 

Yet again, a similar generalization of Theorem \ref{nwfort} was proven to be out of reach in \cite{bikks}.

\begin{Theorem}\label{uplow}
The problem of deciding whether, given a graph $G$ and two functions $\ell_1,\ell_2:V(G) \rightarrow \mathbb{Z}_{\geq 0}$, there is a well-balanced orientation $\vec{G}$ of $G$ such that $\ell_1(v)\leq d_{\vec{G}}^+(v) \leq \ell_2(v)$ for all $v \in V(G)$, is NP-complete.
\end{Theorem}
 A similar result for best-balanced orientations is also proven in \cite{bikks}.
\medskip

In this article, we deal with the question whether a version of the above problem with milder restrictions on the vertex degrees is better tractable. We are interested in the case when instead of imposing an upper and a lower bound on the out-degree of every vertex only an upper bound is imposed. 

More concretely, we consider the following two problems:
\medskip

\noindent {\bf Upper-bounded well-balanced orientation (UBWBO):}
\smallskip

\noindent Input: A graph $G$, a function $\ell:V(G)\rightarrow \mathbb{Z}_{\geq 0}$.

\noindent Question: Is there a well-balanced orientation $\vec{G}$ of $G$ such that $d_{\vec{G}}^+(v)\leq \ell(v)$ for all $v \in V(G)$?
\medskip

\noindent {\bf Upper-bounded best-balanced orientation (UBBBO):}
\smallskip

\noindent Input: A graph $G$, a function $\ell:V(G)\rightarrow \mathbb{Z}_{\geq 0}$.

\noindent Question: Is there a best-balanced orientation $\vec{G}$ of $G$ such that $d_{\vec{G}}^+(v)\leq \ell(v)$ for all $v \in V(G)$?
\medskip

Observe that any orientation obtained from a well-balanced (best-balanced) orientation by reversing the orientation of all arcs is again well-balanced (best-balanced). Hence imposing lower bounds instead of upper bounds on the out-degrees would lead to equivalent problems. Similarly, the bounds could be imposed on the in-degrees instead of the out-degrees.

The question of the complexity of UBBBO can be found in various sources. It is mentioned by Frank in \cite{book}, by Bern\' ath et al. in \cite{bikks} and there is an online posting on it in the open problem collection of the Egerv\' ary Research group \cite{egreswell}. The contribution of this article is to prove that even these problems involving milder restrictions remain hard. We prove the following two results:
\begin{Theorem}\label{well}
UBWBO is NP-complete.
\end{Theorem}

\begin{Theorem}\label{best}
UBBBO is NP-complete.
\end{Theorem}

Observe that Theorem \ref{well} implies Theorem \ref{uplow}. Theorems \ref{well} and \ref{best} can be considered yet another indication of the isolated position that Theorem \ref{nwfort} has in the theory of graph orientations.
\medskip

After a collection of formal definitions and preliminary results in Section \ref{prel}, we prove Theorems \ref{well} and \ref{best} in Section \ref{redu} using a reduction from Cubic Vertex Cover. While our reduction is inspired by the one used in \cite{bikks} to prove Theorem \ref{uplow}, it is more involved.

\section{Preliminaries}\label{prel}

This section is dedicated to providing the background for the proof of the main results in Section \ref{redu}. We first define all important terms in Section \ref{defin} and then give some preliminary results in Section \ref{vorb}.

\subsection{Definitions}\label{defin}

We first give some basic notions of graph theory. A {\it mixed graph} $F$ consists of a vertex set $V(F)$, an edge set $E(F)$, and an arc set $A(F)$. We also say that $F$ {\it contains} $V(F),E(F)$ and $A(F)$. An edge $e=uv \in E(F)$ is a set containing the vertices $u$ and $v$. We say that $e$ {\it links} $u$ and $v$ and $e$ is {\it incident} to $u$ and $v$. More generally, we say that $e$ {\it links} two disjoint sets $X,Y \subseteq V(F)$ if $u \in X$ and $v \in Y$. If $e$ links $X$ and $V(F)-X$, we say that $e$ {\it enters} $X$. An arc $a=uv \in A(F)$ is an ordered tuple of the vertices $u,v \in V(F)$ where $u$ is called the {\it tail} of $a$ and $v$ is called the {\it head} of $a$.  For some $X \subseteq V(F)$ with $u \in X$ and $v \in V(F)-X$, we say that $e$ {\it enters} $V(F)-X$ and {\it leaves} $X$. For some $e=uv \in E(F) \cup A(F)$, we say that $u$ and $v$ are the {\it endvertices} of $e$. A {\it mixed subgraph} $F'$ of $F$ is a mixed graph $F'$ with $V(F')\subseteq V(F), E(F') \subseteq E(F)$, and $A(F') \subseteq A(F)$. For some $X \subseteq V(F)$, we let $F[X]$ denote the mixed subgraph of $F$ whose vertex set is $X$ and that contains all the edges in $E(F)$ and all the arcs in $A(F)$ whose both endvertices are in $X$.

A mixed graph $G$ without arcs is called a {\it graph}. For a graph $G$ and some $X \subseteq V(G)$, we let $d_G(X)$ denote the number of edges in $E(G)$ that have exactly one endvertex in $X$ and we let $i_G(X)$ denote the number of edges in $E(G)$ that have both endvertices in $X$. For a single vertex $v \in V(G)$, we abbreviate $d_G(\{v\})$ to $d_G(v)$ and call this number the {\it degree} of $v$ in $G$.  If $d_G(v)=3$ for all $v \in V(G)$, we say that $G$ is {\it cubic}. For two vertices $u,v \in V(G)$, we use $\lambda_G(u,v)$ for $\min_{u \in X \subseteq V(G)-v}d_G(X)$. Observe that $\lambda_G(u,v) =\lambda_G(v,u)$. For some positive integer $k$, we say that $G$ is {\it $k$-edge-connected} if $\lambda_G(u,v)\geq k$ for all $u,v \in V(G)$. 
 A 1-edge-connected graph which contains two vertices $u,v$ of degree 1 and in which all other vertices are of degree 2 is called a {\it $uv$-path}. We also say that $u$ and $v$ are the {\it endvertices} of the path. Two graphs whose edge sets are disjoint are called {\it edge-disjoint}. For two paths $T_1,T_2$ with $V(T_1)\cap V(T_2)=x$ for a vertex $x$ that is an endvertex of both $T_1$ and $T_2$, we denote by $T_1T_2$ the path with $V(T_1T_2)=V(T_1)\cup V(T_2)$ and $E(T_1T_2)=E(T_1)\cup E(T_2)$.

A mixed graph $D$ without edges is called a {\it digraph}. For a digraph $D$ and some $X \subseteq V(D)$, we let $d_D^+(X)$ denote the number of arcs whose tail is in $X$ and whose head is in $V(D)-X$. We use $d_D^-(X)$ for $d_D^+(V(D)-X)$.  For a single vertex $v \in V(D)$, we abbreviate $d_D^+(\{v\})(d_D^-(\{v\}))$ to $d_D^+(v)(d_D^-(v))$ and call this number the {\it out-degree (in-degree)} of $v$ in $D$. If $d_D^+(v)=d_D^-(v)$ for all $v \in V(D)$, we say that $D$ is {\it eulerian}. Given a function $\ell:V(D)\rightarrow \mathbb{Z}_{\geq 0}$, we say that $D$ is {\it $\ell$-bounded} if $d_D^+(v)\leq \ell(v)$ for all $v \in V(D)$. For two vertices $u,v \in V(D)$, we use $\lambda_D(u,v)$ for $\min_{v \in X \subseteq V(D)-u}d_D^-(X)$. For some positive integer $k$, we say that $D$ is {\it $k$-arc-connected} if $\lambda_D(u,v)\geq k$ for all $(u,v) \in V(D)\times V(D)$. We abbreviate 1-arc-connected to {\it strongly connected.} The operation of exchanging the head and the tail of an arc is called {\it reversing} the arc. Two digraphs whose arc sets are disjoint are called {\it arc-disjoint}.

A mixed graph $F'$ is called a {\it partial orientation} of another mixed graph $F$ if $F'$ can be obtained from $F$ by replacing some of the edges in $E(F)$ by an arc with the same two endvertices. This operation is called {\it orienting} the edge. If $F'$ is a digraph, then $F'$ is called an {\it orientation} of $F$. The unique graph $G$ such that $F$ is an orientation of $G$ is called the {\it underlying graph} of $F$. A strongly connected orientation of a graph all of whose vertices are of degree 2 is called a {\it circuit}. An orientation $T$ of a $uv$-path with $\lambda_T(u,v)=1$ is called a {\it directed $uv$-path.} We say that an orientation $\vec{G}$ of a graph $G$ is {\it well-balanced} if $\lambda_{\vec{G}}(u,v)\geq \lfloor \frac{\lambda_{G}(u,v)}{2}\rfloor $ for all $(u,v) \in V(G)\times V(G)$. If additionally $d_{\vec{G}}^+(v) \in \{\lfloor \frac{d_G(v)}{2}\rfloor,\lceil \frac{d_G(v)}{2}\rceil\}$ holds for all $v \in V(G)$, then $\vec{G}$ is called {\it best-balanced}. We also say that a digraph is {\it well-balanced (best-balanced)} if it is a well-balanced (best-balanced) orientation of its underlying graph.
\medskip

For basic notions of complexity theory, see \cite{gj}. Given a graph $H$, a {\it vertex cover} of $H$ is a subset $U$ of $V(H)$ such that every $e \in E(H)$ is incident to at least one vertex in $U$. We consider the following algorithmic problem:
\medskip

\noindent {\bf Cubic Vertex Cover (CVC):}
\smallskip

\noindent Input: A cubic graph $H$, a positive integer $k$.

\noindent Question: Is there a vertex cover of $H$ of size at most $k$?

\subsection{Preliminary results}\label{vorb}
For proving the correctness of our reduction, we need a few preliminaries. 

The following classic results are due to Menger \cite{men} and fundamental to graph connectivity.

\begin{Theorem}\label{mengerun}
Let $G$ be a graph and $s_1,s_2 \in V(G)$. Then the maximum number of pairwise edge-disjoint $s_1s_2$-paths  in $G$ is $\lambda_G(s_1,s_2)$.
\end{Theorem}

The second result is the directed analogue of Theorem \ref{mengerun}.

\begin{Theorem}\label{menger}
Let $D$ be a digraph and $s_1,s_2 \in V(D)$. Then the maximum number of pairwise arc-disjoint directed $s_1s_2$-paths  in $D$ is $\lambda_D(s_1,s_2)$.
\end{Theorem}


The next result is helpful when proving that a given orientation is well-balanced.

\begin{Proposition}\label{prel3}
Let $G$ be a graph and $a \in V(G)$. Let $\vec{G}$ be an orientation of $G$ such that $\lambda_{\vec{G}}(a,s)\geq \lfloor \frac{d_G(s)}{2}\rfloor$ and $\lambda_{\vec{G}}(s,a)\geq \lfloor \frac{d_G(s)}{2}\rfloor$ hold for all $s \in V(G)-a$. Then $\vec{G}$ is well-balanced.
\end{Proposition}

\begin{proof}
Let $s_1,s_2 \in V(G)$ and $R \subseteq V(G)-s_1$ with $s_2 \in R$. If $a \in R$, we have $d_{\vec{G}}^-(R)\geq \lambda_{\vec{G}}(s_1,a)\geq \lfloor\frac{d_G(s_1)}{2} \rfloor\geq \lfloor \frac{\lambda_{G}(s_1,s_2)}{2}\rfloor$. If $a \in V(G)- R$, we have $d_{\vec{G}}^-(R)\geq \lambda_{\vec{G}}(a,s_2)\geq \lfloor\frac{d_G(s_2)}{2} \rfloor\geq \lfloor \frac{\lambda_{G}(s_1,s_2)}{2}\rfloor$. In either case, we obtain $d_{\vec{G}}^-(R)\geq \lfloor \frac{\lambda_{G}(s_1,s_2)}{2}\rfloor$, so $\lambda_{\vec{G}}(s_1,s_2)\geq \lfloor \frac{\lambda_{G}(s_1,s_2)}{2}\rfloor$. Hence $\vec{G}$ is well-balanced.
\end{proof}



The next simple result allows to modify orientations maintaining important properties.

\begin{Proposition}\label{prel5}
Let $G$ be a graph, $\ell: V(G)\rightarrow \mathbb{Z}_{\geq 0}$ a function, $\vec{G}_0$ an $\ell$-bounded, well-balanced orientation of $G$, $D$ an eulerian directed subgraph of $\vec{G}_0$ and $\vec{G}_1$  the orientation of $G$ which is obtained by reversing all the arcs of $D$. Then $\vec{G}_1$ is $\ell$-bounded and well-balanced.
\end{Proposition}

\begin{proof}
Since $D$ is eulerian, we have $d_{\vec{G}_1}^+(s)=d_{\vec{G}_0}^+(s)$ for all $s \in V(G)$. Hence, as $\vec{G}_0$ is $\ell$-bounded, so is $\vec{G}_1$. Similarly, we have $d_{\vec{G}_1}^-(s)=d_{\vec{G}_0}^-(s)$ for all $s \in V(G)$.
We hence have 
$d_{\vec{G}_1}^-(R)=\sum_{s \in R}d_{\vec{G}_1}^-(s)-i_G(R)=\sum_{s \in R}d_{\vec{G}_0}^-(s)-i_G(R)=d_{\vec{G}_0}^-(R)$ for all $R \subseteq V$. Hence $\lambda_{\vec{G}_1}(s_1,s_2)=\min_{s_2\in R \subseteq V-s_1}d_{\vec{G}_1}^-(R)=\min_{s_2\in R \subseteq V-s_1}d_{\vec{G}_0}^-(R)=\lambda_{\vec{G}_0}(s_1,s_2)$ for all $(s_1,s_2) \in V(G)\times V(G)$.
Thus, as $\vec{G}_0$ is well-balanced, so is $\vec{G}_1$.
\end{proof}

Finally, we need the following result to justify the usefulness of our reduction. It can be found in \cite{gj}.

\begin{Theorem}\label{vcdure}
Cubic Vertex Cover is NP-complete.
\end{Theorem} 

\section{The reduction}\label{redu}

In this section, we give the reduction we need to prove Theorems \ref{well} and \ref{best}. We first give a reduction for Theorem \ref{well} and then show how to adapt it to prove Theorem \ref{best}. In Section \ref{def}, we describe the instance $(G,\ell)$ of UBWBO we create from a given instance $(H,k)$ of CVC. In the remaining part of the paper $(H,k)$ and $(G,\ell)$ are fixed. In Section \ref{convor}, we describe a particular kind of orientations, called convenient orientations that play a crucial role in the proof of the reduction. In Section \ref{vcwell}, we give the first direction of the reduction showing how to obtain an $\ell$-bounded, well-balanced orientation of $G$ from a vertex cover of $H$. The other direction is divided in two parts. First, we show in Section \ref{faireconv} how an $\ell$-bounded, well-balanced orientation of $G$ can be turned into one that additionally has the property of being convenient. After, in Section \ref{tocov}, we show how an orientation with this extra property yields a vertex cover of $H$. In Section \ref{bbo}, we show how to adapt our construction for the proof of Theorem \ref{best}. Finally, in Section \ref{conc}, we conclude our proof.

\subsection{The construction}\label{def}

We here show how to create an instance of UBWBO from an instance of CVC. Let $(H,k)$ be an instance of CVC. Since $H$ is cubic, we have $|V(H)|=2n$ and $|E(H)|=3n$ for some integer $n\ge 2.$

We first describe, for every $v \in V(H)$,  a vertex gadget $G^v$ that contains 6 vertices: $p_0^v, p_1^v, p^v_2,q_0^v,q_1^v, q^v_2$ and 5 edges:  $p_0^vp_1^v,p_1^vp_2^v,q_0^v q_1^v,q_1^v q_2^v, p^v_0q_0^v$. 
We next describe,  for every $e \in E(H)$, an edge gadget $G^{e}$ that contains 6 vertices $x^{e},y^{e},z^{e}_1,z^{e}_2,z^{e}_3,z^{e}_4$
and 5 edges: $x^{e}y^{e},x^{e}z^{e}_1, y^{e}z^{e}_2,y^{e}z^{e}_3, y^{e}z^{e}_4$. 
An illustration of these gadgets can be found in Figure \ref{dad2}.

\begin{figure}[h]\begin{center}
  \includegraphics[]{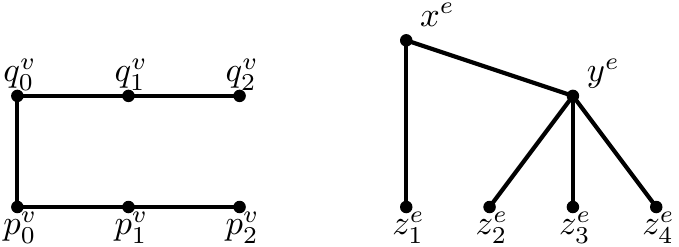} 
  \caption{A vertex gadget for a vertex $v$ and an edge gadget for an edge $e$.}\label{dad2}
\end{center}
\end{figure}

 \medskip

We are now ready to describe $G$. For every $v \in V(H)$, we let $G$ contain a vertex gadget $G^v$ and for every $e \in E(H)$, we let $G$ contain an edge gadget $G^e$. Let $P=\bigcup_{v \in V(H)}V(G^v), X=\bigcup_{e \in E(H)}\{x^{e},y^{e}\}$ and 
$Z=\bigcup_{e \in E(H)}\{z^e_1,z^e_2,z^e_3,z^e_4\}.$
We let $V(G)$ contain two more vertices $a$ and $b$. We now finish the description of $G$ by linking these components by some additional edges. For every $z \in Z$, we let $E(G)$ contain an edge $az$ and an edge $bz$. Further, for every $v \in V(H)$, let $e_1,e_2,e_3$ be an arbitrary ordering of the edges in $E(H)$ which are incident to $v$ in $H$. We add the edges $ap^v_0,aq^v_0,p_1^{v}y^{e_1}, p_2^vy^{e_2}, p_2^vy^{e_3},q_1^{v}x^{e_1}, q_2^vx^{e_2},$ and $q_2^vx^{e_3}$.  This finishes the construction of $G$.

 Observe that $d_G(a)=4|E(H)|+2|V(H)|=16n, d_G(b)=4|E(H)|=12n, d_G(s)=3$ for all $s \in P \cup Z$, and $d_G(x^{e})=4$ and $d_G(y^{e})=6$ for all $e \in E(H)$. An illustration can be found in Figure \ref{dad3}.

\begin{figure}[h]\begin{center}
  \includegraphics[]{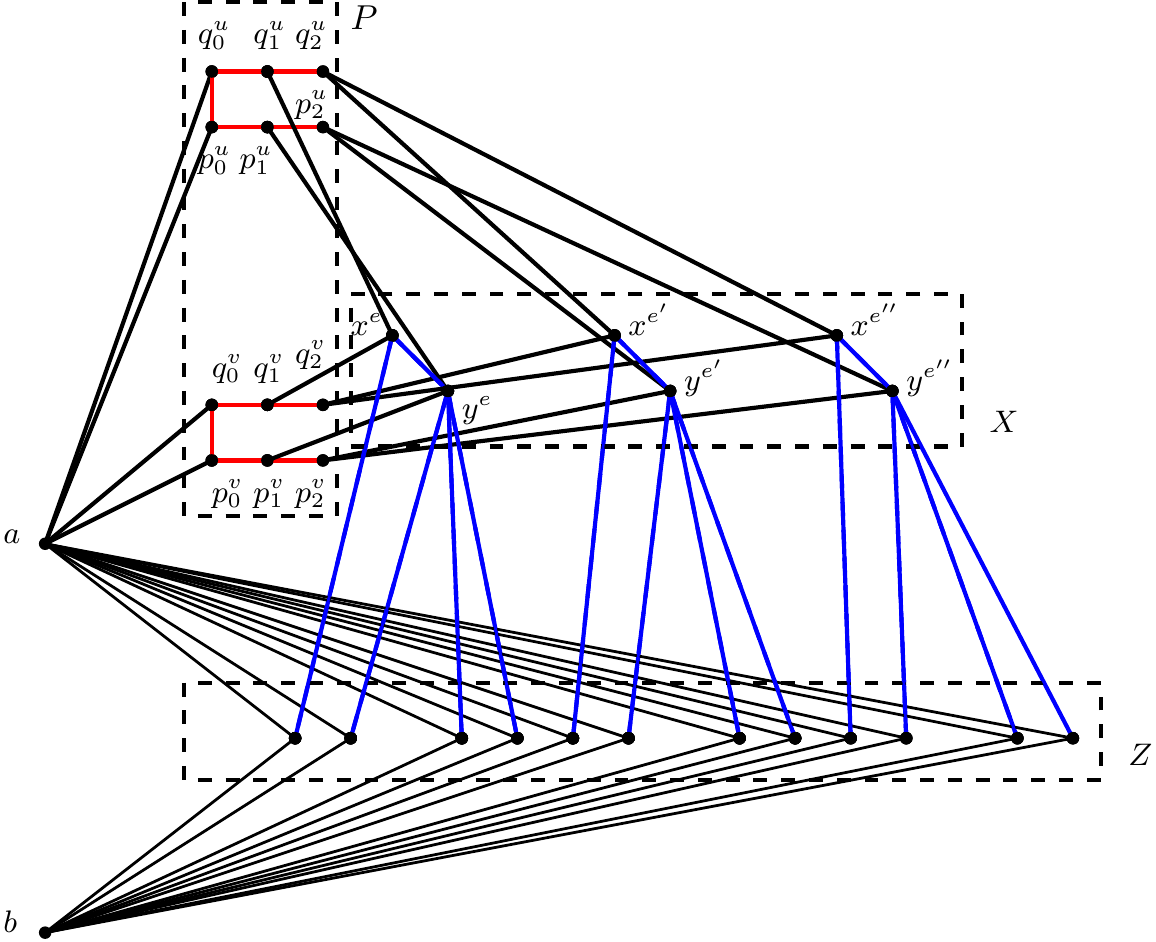} 
  \caption{An example for the graph $G$ created from a graph $H$ where $V(H)$ contains two vertices $u$ and $v$ and $E(H)$ contains three parallel edges $e,e'$, and $e''$ linking $u$ and $v$. All the edges belonging to a vertex gadget are marked in red while all the edges belonging to an edge gadget are marked in blue. The names of the vertices in $Z$ have been omitted due to space restrictions. They are from left to right: $z^{e}_1,z^{e}_2,z^{e}_3,z^{e}_4,z^{e'}_1,z^{e'}_2,z^{e'}_3,z^{e'}_4,z^{e''}_1,z^{e''}_2,z^{e''}_3,z^{e''}_4$. See also Figure \ref{dad2}.}\label{dad3}
\end{center}
\end{figure}

We now define $\ell$. We set $\ell(a)=8n+k$ and $\ell(z)=1$ for all $z \in Z$. For all $s \in V(G)-(Z \cup a)$, we set the trivial bound $\ell(s)=d_G(s)$.
\medskip

We now give an important result on the connectivity properties of $G$.

\begin{Proposition}\label{wbd}
$\lambda_{G}(s,a)=d_G(s)$ for all $s \in V(G)-a$.
\end{Proposition}
\begin{proof}
By definition, $\lambda_{G}(s,a)\le d_G(s)$ for all $s \in V(G)-a$.
First observe that for every $z \in Z$, $G$ contains the $ba$-path $bza$. By Theorem \ref{mengerun}, this yields $\lambda_G(b,a)\geq |Z|=12n=d_G(b)$. Now consider some $e=uv \in E(H)$. By construction, there are some $i,j \in \{1,2\}$ such that $G$ contains the edges $q_i^{u}x^{e},p_i^{u}y^{e},q_j^{v}x^{e}$ and $p_j^{v}y^{e}$. Due to the pairwise edge-disjoint $x^{e}a$-paths $T_1=x^{e}z^{e}_1a,T_2=x^{e}y^{e}z^{e}_2a,T_3= x^{e}q_i^{u}\ldots q_0^{u}a$ and $T_4= x^{e}q_j^{v}\ldots q_0^{v}a$ and Theorem \ref{mengerun}, we obtain that $\lambda_G(x^{e},a)\geq 4=d_G(x^{e})$. Due to the pairwise edge-disjoint $y^{e}a$-paths $T_1=y^{e}x^{e}z^{e}_1a,T_2=y^{e}z^{e}_2a,T_3= y^{e}z^{e}_3a,T_4= y^{e}z^{e}_4a, T_5=y^{e}p_i^{u}\ldots p_0^{u}a$ and $T_6=y^{e}p_j^{v}\ldots p_0^{v}a$ and Theorem \ref{mengerun}, we obtain that $\lambda_G(y^{e},a)\geq 6=d_G(y^{e})$. 
 Finally, suppose for the sake of a contradiction that  for some $t\in P\cup Z,$ $\lambda_G(t,a)<d_G(t)=3$. Let $a\in R \subseteq V(G)-t$ with $d_G(R)=\lambda_G(t,a)$. As $\lambda_G(s,a)\geq 3$ for all $s \in X \cup b$, we obtain $X \cup b \subseteq R$. Hence, since every $z \in Z$ is adjacent to three vertices in $R$, we obtain $Z \subseteq R$, so $t\in P.$  As $t$ is adjacent to three distinct vertices in $G$ and every vertex in $P$ is linked to $R$, we obtain $d_G(R)\geq 3$, a contradiction. 
\end{proof}
By Propositions \ref{prel3} and \ref{wbd}, we have the following characterization of well-balanced orientations of $G.$

\begin{Corollary}\label{ggvftu}
Let $\vec{G}$ be an orientation of $G.$
	\begin{itemize}
		\item[(a)] $\vec{G}$ is well-balanced if and only if $\lambda_{\vec{G}}(a,s)\geq \lfloor\frac{d_G(s)}{2}\rfloor$ and $\lambda_{\vec{G}}(s,a)\geq \lfloor\frac{d_G(s)}{2}\rfloor$ for all $s \in V(G)-a$.
		\item[(b)] If $\vec{G}$ is well-balanced, then $d^+_{\vec{G}}(s)=d^-_{\vec{G}}(s)=\frac{d_G(s)}{2}=\lambda_{\vec{G}}(a,s)=\lambda_{\vec{G}}(s,a)$   for all $s \in X\cup b$.
\end{itemize}
\end{Corollary}
\begin{proof}
 $(a)$ The sufficiency is Proposition \ref{prel3}. The necessity is an immediate consequence of the definition of well-balanced orientations and Proposition \ref{wbd}. 

$(b)$ Suppose that $\vec{G}$ is well-balanced and let $s  \in X \cup b$. As $d_G(s)$ is even and by $(a)$, we have $d_G(s)=2\lfloor\frac{d_{G}(s)}{2}\rfloor\leq \lambda_{\vec{G}}(s,a)+\lambda_{\vec{G}}(a,s)\leq d^+_{\vec{G}}(s)+d^-_{\vec{G}}(s)=d_G(s)$, hence equality holds throughout.
\end{proof}

\subsection{Convenient orientations}\label{convor}

In order to prove that the reduction works indeed, we wish to consider a certain restricted class of orientations. We now define a mixed graph $F$  which  is obtained as a partial orientation of $G$.

 First for every $e \in E(H)$ and $i \in \{1,2\}$, let the edge $az^{e}_i$ be oriented from $a$ to $z^{e}_i$ and the edge $bz^{e}_i$ be oriented from $z^{e}_i$ to $b$. For every $e \in E(H)$ and $i \in \{3,4\}$, let the edge $az^{e}_i$ be oriented from $z^{e}_i$ to $a$ and the edge $bz^{e}_i$ be oriented from $b$ to $z^{e}_i$. Let all the edges linking $X$ and $Z$ be oriented from $X$ to $Z$. For every $e \in E(H)$, let the edge $x^{e}y^{e}$ be oriented from $x^{e}$ to $y^{e}$. Next, let all the edges linking $P$ and $X$ be oriented from $P$ to $X$.  For every $v \in V(H)$ and $i \in \{0,1\}$, let the edge $p^v_{i}p^v_{i+1}$ be oriented from $p^v_{i}$ to $p^v_{i+1}$ and  let the edge $q^v_{i}q^v_{i+1}$ be oriented from $q^v_{i}$ to $q^v_{i+1}$. We denote the obtained partial orientation of $G$ by $F$. 
 Observe that the edge set of $F$ consists  of the 3 edges $aq^v_0,ap^v_0,p^v_0q^v_0$ for every $v \in V(H)$. 
 An illustration of $F$ can be found in Figure \ref{dad4}.

\begin{figure}[h]\begin{center}
  \includegraphics[]{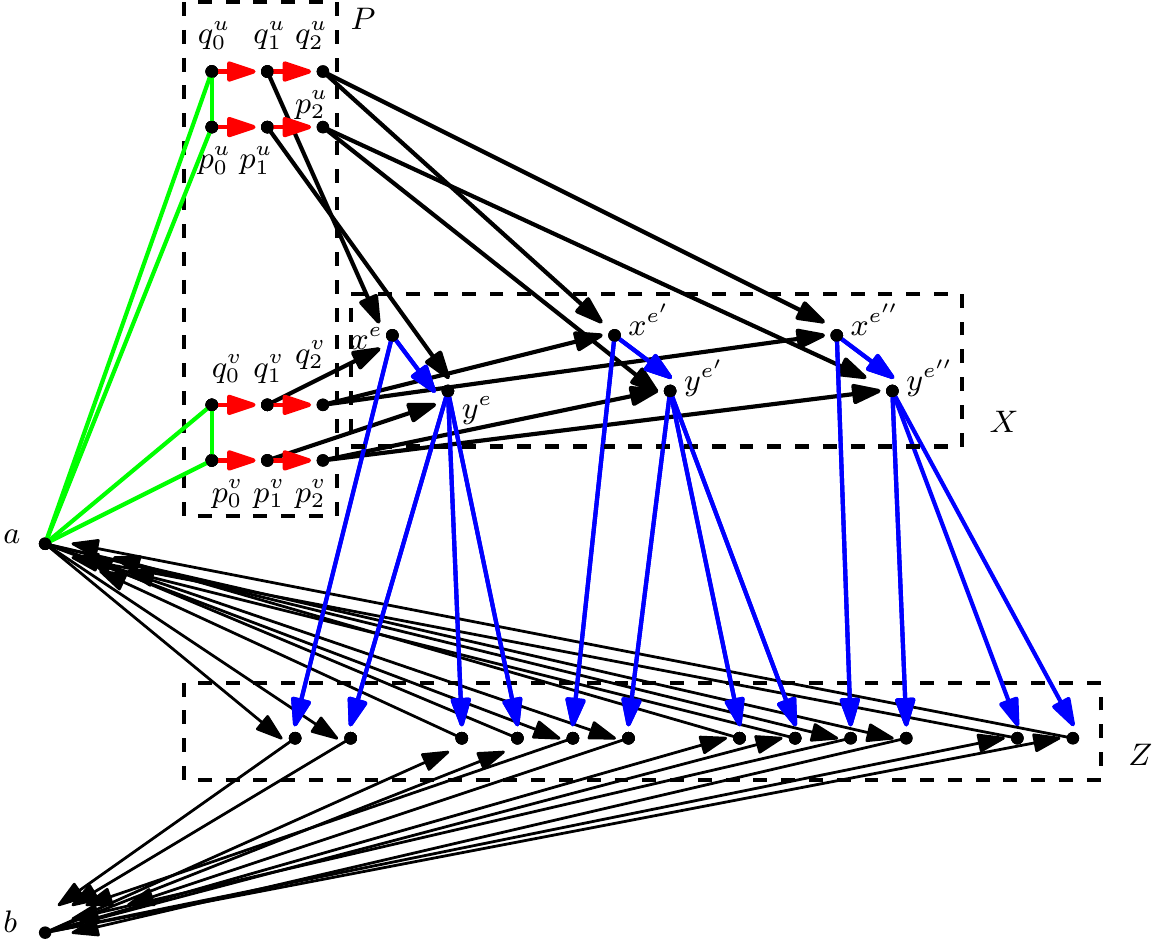} 
  \caption{An example for the mixed graph $F$ created from the same graph $H$  as considered in Figure \ref{dad3}. The edges of $F$ are marked in green.}\label{dad4}
\end{center}
\end{figure}

We now say that an orientation $\vec{G}$ of $G$ is {\it convenient} if $\vec{G}$ is also an orientation of $F$. The following  lemma contains a characterization of convenient, well-balanced orientations of $G$ which is a crucial ingredient for proving the correctness of our reduction.

\begin{Lemma}\label{vfasf}
A convenient orientation $\vec{G}$ of $G$ is well-balanced if and only if for every $uv \in E(H)$, 

(i) either the edges from $a$ to $\{p^u_0,q^u_0\}$ are oriented from $a$ to $\{p^u_0,q^u_0\}$ or the edges from $a$ to $\{p^v_0,q^v_0\}$ are oriented from $a$ to $\{p^v_0,q^v_0\}$,

(ii) in $\vec{G}[\{a,p^u_0,q^u_0,p^v_0,q^v_0\}]$, there is a directed $as$-path  for all $s \in\{p^u_0,q^u_0,p^v_0,q^v_0\}$. 
\end{Lemma}
\begin{proof}
First suppose that $\vec{G}$  is well-balanced and let $e=uv \in E(H)$. Consider the set $R=V(G^{u})\cup V(G^{v}) \cup \{x^{e},y^{e}\}$. By $y^e\in R\subseteq V-a$ and Corollary \ref{ggvftu}(b), we have $d_{\vec{G}}^-(R)\geq \lambda_{\vec{G}}(a,y^{e})=\frac{d_G(y^{e})}{2}=3$. As $\vec{G}$ is convenient, it follows that the only arcs entering $R$ in $\vec{G}$ have the tail $a$. Since  the set of edges linking $a$ and $R$ consists of the four edges $ap^u_0,aq^u_0,ap^v_0,$ and $aq^v_0$, we get that either the arcs $ap^u_0$ and $aq^u_0$ exist in $\vec{G}$ or the arcs $ap^v_0$ and $aq^v_0$ exist in $\vec{G}$, that is (i) holds. 

Consider a vertex $s$ in $\{p^u_0,q^u_0,p^v_0,q^v_0\}.$ Since $\vec{G}$  is well-balanced, by Corollary \ref{ggvftu}(a), $d_G(s)=3$ and  Theorem \ref{menger}, there exists a directed $as$-path in $\vec{G}$. Since $\vec{G}$ is convenient, this path also exists in $\vec{G}[\{a,p^u_0,q^u_0,p^v_0,q^v_0\}]$, that is (ii) holds.


\medskip

  For the other direction, we will use that $\vec{G}$ is convenient several times without explicit mention. 
  By Corollary \ref{ggvftu}(a), it suffices to prove that $\lambda_{\vec{G}}(a,s)\geq \lfloor\frac{d_G(s)}{2}\rfloor$ and $\lambda_{\vec{G}}(s,a)\geq \lfloor\frac{d_G(s)}{2}\rfloor$ for all $s \in V(G)-a$.


First we consider $b$. For every $e \in E(H)$, $i \in\{1,2\}$ and $j\in\{3,4\}$,  $a z^{e}_ib$ and $bz^e_ja$ is a directed $ab$-path and $ba$-path, respectively, in $\vec G$. We obtain, by Theorem \ref{menger}, that $\lambda_{\vec{G}}(a,b),$ $\lambda_{\vec{G}}(b,a) \geq 2|E(H)| =\lfloor\frac{d_G(b)}{2}\rfloor$. 

We next consider the vertices in $X.$
Let $e=uv \in E(H)$. By construction, there are indices $i , j\in \{1,2\}$ such that $\vec{G}$ contains the arcs $p^{u}_{i}y^{e}, q^{u}_{i}x^{e},p^{v}_{j}y^{e}$, and $q^{v}_{j}x^{e}$. By (ii), there is a directed $as$-path $T_s$ in $\vec{G}[\{a,p^u_0,q^u_0,p^v_0,q^v_0\}]$  for every $s \in \{p^u_0,q^u_0,p^v_0,q^v_0\}$.

Since $T_1=T_{q_0^{u}}q_0^{u}q_1^u\ldots q_i^ux^{e}$ and $T_2=T_{q_0^v}q_0^vq_1^v\ldots q_i^vx^{e}$ are two arc-disjoint directed $ax^{e}$-paths, by Theorem \ref{menger}, we obtain $\lambda_{\vec{G}}(a,x^{e})\geq 2= \lfloor\frac{d_G(x^{e})}{2}\rfloor$.
Since  $T_1=x^{e}z^{e}_1bz^{e}_3a$ and $T_2=x^{e}y^{e}z^{e}_4a$ are two arc-disjoint directed $x^{e}a$-paths, by Theorem \ref{menger}, we obtain $\lambda_{\vec{G}}(x^{e},a)\geq 2= \lfloor\frac{d_G(x^{e})}{2}\rfloor$.

Next consider $T_1=y^{e}z^{e}_2bz^{e'}_3a,T_2=y^{e}z^{e}_3a$ and $T_3=y^{e}z^{e}_4a$, where $e' \in E(H)-e$ is chosen arbitrarily. These are three arc-disjoint directed $y^{e}a$-paths, so by Theorem \ref{menger}, we have $\lambda_{\vec{G}}(y^{e},a)\geq 3= \lfloor\frac{d_G(y^{e})}{2}\rfloor$.
For the next part, by (i) and symmetry, we may suppose that the arcs $ap^u_0,aq^u_0$ exist in $\vec{G}$. Since  $T_1=ap_0^u\ldots p_i^uy^{e}, T_2=aq_0^u\ldots q_i^uy^{e}$, and $T_3=T_{p_0^v}p_1^v\ldots p_j^vy^{e}$ are three arc-disjoint directed $ay^{e}$-paths, by Theorem \ref{menger}, we obtain $\lambda_{\vec{G}}(a,y^{e})\geq 3= \lfloor\frac{d_G(y^{e})}{2}\rfloor$.

Let $R$ be the vertex set of the strongly connected component of $\vec{G}$ containing $a$. By the above, we have $X \cup b \subseteq R$. Next, every $z \in Z$ is incident to an arc entering $R-z$ and an arc leaving $R-z$, so $Z \subseteq R$. Now let $v \in V(H)$. For every $p \in V(G^v)$, by $(ii)$, a directed $ap$-path is contained in $\vec{G}[V(G^v)\cup a]$. Further, $\vec{G}$ contains a directed path from $p$ to $X$. We hence obtain that $P \subseteq R$, so $\vec{G}$ is strongly connected. 
This yields $\lambda_{\vec{G}}(a,s) \geq 1 =\lfloor\frac{d_G(s)}{2}\rfloor$ and $\lambda_{\vec{G}}(s,a) \geq 1 =\lfloor\frac{d_G(s)}{2}\rfloor$ for all $s \in P \cup Z$.
\end{proof}



\subsection{From vertex cover to orientation}\label{vcwell}
In this section, we give the first direction of the reduction. More formally, we prove the following result.
\begin{Lemma}\label{covtowell}
If there exists a vertex cover of size at most $k$ of $H$, then there exists an $\ell$-bounded, well-balanced orientation of $G$.
\end{Lemma}
\begin{proof}
Let $U$ be a vertex cover of size at most $k$ of $H$. Let $\vec{G}$ be the unique convenient orientation of $G$ in which for every $v \in V(H)$, the edges $ap_0^v,p_0^vq_0^v$ are oriented to a directed path   $ap_0^vq_0^v$; further the edge $aq_0^v$ is oriented from $a$ to $q_0^v$ if and only if $v \in U$. By Lemma \ref{vfasf}  and as $U$ is a vertex cover, we obtain that $\vec{G}$ is well-balanced. By construction, we have $d_{\vec{G}}^+(s)\leq \ell(s)$ for all $s \in V(G)-a$. Finally, $\vec{G}$ contains $2|E(H)|$ arcs from $a$ to $Z$, one arc from $a$ to $p_0^v$ for all $v \in V(H)$ and one arc from $a$ to $q^v_0$ for all $v \in U$. This yields $d_{\vec{G}}^+(a)=2|E(H)|+|V(H)|+|U|\leq 6n+2n+k=\ell(a)$, so $\vec{G}$ is $\ell$-bounded.
\end{proof}

\subsection{Making a well-balanced orientation convenient}\label{faireconv}
In this section, we give a slightly technical lemma that shows that if an $\ell$-bounded, well-balanced orientation of $G$ exists, we can also find one which is additionally convenient.

\begin{Lemma}\label{sconv}
If there exists a well-balanced, $\ell$-bounded orientation of $G$, then there also exists a convenient, well-balanced, $\ell$-bounded orientation of $G$.
\end{Lemma}

\begin{proof}
Let $\vec{G}_0$ be a well-balanced, $\ell$-bounded orientation of $G.$

Let $Z_0^+$ be the set of all $z \in Z$ such that $\vec{G}_0$ contains the arc $bz$ and let $Z_0^-=Z-Z_0^+$. As $\vec{G}_0$ is well-balanced and by Corollary \ref{ggvftu} $(b)$, we have $|Z_0^+|=|Z_0^-|=6n$. Further, let $Z_0^*$ be the set of all $z \in Z$ such that $\vec{G}_0$ contains an arc from $z$ to $X$. Observe that $Z_0^* \subseteq Z_0^+$ because $\vec{G}_0$ is $\ell$-bounded.

\begin{Claim}
There is a set of pairwise arc-disjoint circuits $\{C_z:z \in Z_0^*\}$ such that $V(C_z)\cap Z =z$ for all $z \in Z_0^*$.
\end{Claim}

\begin{proof}
By Corollary \ref{ggvftu}(b), we have $\lambda_{\vec{G}_0}(b,a)=d_{\vec{G}_0}^+(b)=\frac{d_G(b)}{2}=6n$.
By Theorem \ref{menger}, there is a set $\mathcal{T}$ of $6n$ pairwise arc-disjoint directed $ba$-paths in $\vec{G}_0$.
For all $z \in Z_0^-$, as $\vec{G}_0$ is $\ell$-bounded and contains the arc $zb$, we obtain that $z$ is not contained in a  directed $ba$-path of $\mathcal{T}$. 

Clearly, every $T \in \mathcal{T}$ contains a vertex in $Z_0^+$. Further, as $\vec{G}_0$ is $\ell$-bounded and the directed $ba$-paths in $\mathcal{T}$ are pairwise arc-disjoint, no vertex in $Z_0^+$ can be contained in two distinct $ba$-paths of $\mathcal{T}$. As $|Z_0^+|=6n=|\mathcal{T}|$, we obtain that every $z \in Z_0^+$ is contained in exactly one directed $ba$-path $T_z$ of $\mathcal{T}$ and $T_z$ satisfies $V(T_z)\cap Z =z$. For every $z \in Z_0^*$, as $\vec{G}_0$ is $\ell$-bounded, the arc $az$ is contained in $\vec{G}_0$. Now let $C_z$ be obtained from $T_z$ by deleting the arc $bz$ and adding the arc $az$. Then $C_z$ is a circuit. Since the directed $ba$-paths in $\mathcal{T}$ are arc-disjoint, $\{C_z:z \in Z_0^*\}$ has the desired properties.
\end{proof}

Let $\vec{G}_1$ be obtained from $\vec{G}_0$ by reversing all the arcs of $\cup_{z \in Z_0^*}A(C_z)$. Observe that in $\vec{G}_1$ all the edges linking $Z$ and $X$ are oriented from $Z$ to $X$. Further observe that for all $z \in Z_0^+$, $\vec{G}_1$ contains the directed $ba$-path $bza$ and for all $z \in Z_0^-$, $\vec{G}_1$ contains the directed $ab$-path $azb$. Now let  $Z^{1,2}=\bigcup_{e \in E(H)}\{z^e_1,z^e_2\}$ and $Z^{3,4}=Z-Z^{1,2}$. Let $D$ be the spanning directed subgraph of $\vec{G}_1$ whose arc set is $\bigcup_{z \in Z_0^+ \cap Z^{1,2}}\{bz,za\} \cup \bigcup_{z \in Z_0^- \cap Z^{3,4}}\{az,zb\}$.

\begin{Claim}
$D$ is eulerian.
\end{Claim}
\begin{proof}
Clearly, we have $d_D^+(s)=d_D^-(s)$ for all $s \in V(G)-\{a,b\}$. Further, we have $d_D^+(b)=|Z_0^+\cap Z^{1,2}|=|Z^{1,2}|-|Z_0^-\cap Z^{1,2}|=6n-|Z_0^-\cap Z^{1,2}|=|Z_0^-|-|Z_0^-\cap Z^{1,2}|=|Z_0^-\cap Z^{3,4}|=d_D^-(b)$ and similarly $d_D^+(a)=d_D^-(a)$.
\end{proof}

Let $\vec{G}_2$ be obtained from $\vec{G}_1$ by reversing the orientation of every arc of $D$.  

Observe that all the edges in $G$ incident to a vertex of $Z$ have the same orientation in $\vec{G}_2$ and $F$. Applying Proposition \ref{prel5} twice, we obtain that $\vec{G}_2$ is well-balanced and $\ell$-bounded. In order to complete the proof of Lemma \ref{sconv}, we  show in the following that $\vec{G}_2$ is convenient.

\begin{Claim}\label{ptox} 
All the edges in $G$ incident to at least one vertex in $X$ have the same orientation in $\vec{G}_2$ and $F$. 
\end{Claim}

\begin{proof}Let $e \in E(H)$.
As observed above, all the edges linking $\{x^{e},y^{e}\}$ and $Z$ are oriented from $\{x^{e},y^{e}\}$ to $Z$ in $\vec{G}_2$. By Corollary \ref{ggvftu}(b), we obtain $d_{\vec{G}_2}^+(y^{e})=d_{\vec{G}_2}^-(y^{e})=\frac{d_G(y^{e})}{2}=3$ and $d_{\vec{G}_2}^+(x^{e})=d_{\vec{G}_2}^-(x^{e})=\frac{d_G(x^{e})}{2}=2$. As $\vec{G}_2$ contains $3$ arcs from $y^{e}$ to $Z$, we obtain that the edges linking $P$ and $y^{e}$ are oriented from $P$ to $y^{e}$ in $\vec{G}_2$ and that the edge $x^{e}y^{e}$ is oriented from $x^{e}$ to $y^{e}$ in $\vec{G}_2$.
As $\vec{G}_2$ contains two arcs from $x^{e}$ to $Z \cup y^{e}$, we obtain that the edges linking $P$ and $x^{e}$ are oriented from $P$ to $x^{e}$ in $\vec{G}_2$.
\end{proof}

\begin{Claim}\label{paths}
For every $v \in V(H)$, the edges in $E(G^v)-\{p_0^vq_0^v\}$ have the same orientation in $\vec{G}_2$ and $F$.
\end{Claim}
\begin{proof}
For every $v \in V(H)$, as $\vec{G}_2$ is well-balanced and by Corollary \ref{ggvftu}(a), we have  $\lambda_{\vec{G}_2}(a,p_2^v)\geq \lfloor\frac{d_G(p_2^v)}{2}\rfloor=1$. Hence, by construction and Claim \ref{ptox}, we obtain that  there is a directed $p_0^vp_2^v$-path  in $\vec{G}_2$, namely $p_0^vp_1^vp_2^v$. Similarly, $q_0^vq_1^vq_2^v$ is a directed $q_0^vq_2^v$-path in $\vec{G}_2$. 
\end{proof}

Claims \ref{ptox} and \ref{paths} finish the proof of the fact that $\vec{G}_2$ is convenient. 
\end{proof}

\subsection{From convenient orientation to vertex cover}\label{tocov}

We now give the last step of the other direction of our reduction. More formally, we prove the following result.

\begin{Lemma}\label{convtocov}
If there is a convenient, well-balanced, $\ell$-bounded orientation of $G$, then there is a vertex cover of size at most $k$ of $H$.
\end{Lemma}

\begin{proof}
Let $\vec{G}$ be a convenient, well-balanced, $\ell$-bounded orientation of $G$. Let $U \subseteq V(H)$ be the set of vertices $v$ for which the arcs $ap_0^v$ and $aq_0^v$ exist in $\vec{G}$. By Lemma \ref{vfasf}, we get that $U$ is a vertex cover of $H$ and for every $v\in V(H),$ at least one arc exists in $\vec{G}$ from $a$ to $V(G_v)$. Next note that there are exactly $2|E(H)|=6n$ arcs leaving $a$ in $F$. As $\vec{G}$ is a convenient, $\ell$-bounded orientation of $G$, we have $|U|=d_{\vec{G}}^+(a)-d_{F}^+(a)-|V(H)|\leq \ell(a)-6n-2n=(8n+k)-8n=k$.
\end{proof}

\subsection{Best-balanced orientations}\label{bbo}

We now show how to extend our reduction to best-balanced orientations. We create an instance $(G',\ell')$ of UBBBO by altering the instance $(G,\ell)$ of UBWBO created in Section \ref{def}. Let $G'$ be obtained from $G$ by adding a set $W$ of $2k$ new vertices and an edge $wa$ for all $w \in W$. Observe that $d_{G'}(a)=d_{G}(a)+|W|=16n+2k$ and $d_{G'}(w)=1$ for all $w\in W$. Further, we set $\ell'(z)=1$ for all $z \in Z$ and we set the trivial bound $\ell'(s)=d_{G'}(s)$ for all $s \in V(G')-Z$.
\begin{Lemma}\label{welltobest}
There exists an $\ell'$-bounded, best-balanced orientation of $G'$ if and only if there exists an $\ell$-bounded, well-balanced orientation of $G$.
\end{Lemma}
\begin{proof}
First suppose that there exists an $\ell'$-bounded, best-balanced orientation $\vec{G'}$ of $G'$. Let $\vec{G}=\vec{G'}[V(G)]$. Observe that $\vec{G}$ is an orientation of $G$. Further, as $\vec{G'}$ is well-balanced, for any $(s_1,s_2) \in V(G)\times V(G)$, we have $\lambda_{\vec{G}}(s_1,s_2)=\lambda_{\vec{G'}}(s_1,s_2)\geq \lfloor \frac{\lambda_{G'}(s_1,s_2)}{2}\rfloor = \lfloor \frac{\lambda_{G}(s_1,s_2)}{2}\rfloor$, hence $\vec{G}$ is well-balanced. For any $s \in V(G)-a$,  as $\vec{G'}$ is $\ell'$-bounded, we have $d^+_{\vec{G}}(s)=d^+_{\vec{G'}}(s)\leq \ell'(s)=\ell(s)$. Finally, as $\vec{G'}$ is best-balanced, we have $d_{\vec{G}}^+(a)\leq d_{\vec{G'}}^+(a)\leq \lceil \frac{d_{G'}(a)}{2}\rceil =8n+k = \ell(a)$. Hence $\vec{G}$ is $\ell$-bounded.

Now suppose that there is an $\ell$-bounded, well-balanced orientation of $G$. We obtain by Lemma \ref{sconv} that there is also a  convenient, $\ell$-bounded, well-balanced orientation $\vec{G}$ of $G$. This yields $8n\leq d_{\vec{G}}^+(a)\leq 8n+k$.


 We now create an orientation $\vec{G'}$ by giving every edge in $E(G)$ the orientation it has in $\vec{G}$, orienting  $8n+k- d_{\vec{G}}^+(a)$ of the edges linking $W$ and $a$ from $a$ to $W$ and orienting all the remaining edges linking $W$ and $a$ from $W$ to $a$. For any $(s_1,s_2) \in V(G)\times V(G)$, we have $\lambda_{\vec{G'}}(s_1,s_2)=\lambda_{\vec{G}}(s_1,s_2)\geq \lfloor \frac{\lambda_{G}(s_1,s_2)}{2}\rfloor = \lfloor \frac{\lambda_{G'}(s_1,s_2)}{2}\rfloor$. For any $(s_1,s_2) \in V(G')\times V(G')$ with $\{s_1,s_2\}\cap W \neq \emptyset$, we have $\lambda_{\vec{G'}}(s_1,s_2) \geq 0= \lfloor \frac{\lambda_{G'}(s_1,s_2)}{2}\rfloor$. Hence, $\vec{G'}$ is well-balanced. For every $s \in V(G)-a$, we have $d^+_{\vec{G'}}(s)=d^+_{\vec{G}}(s)\leq \ell(s)=\ell'(s)$.
 Further, as $\vec{G}$ is convenient, we have $d^+_{\vec{G'}}(s) \in \{\lfloor \frac{d_G(s)}{2}\rfloor, \lceil \frac{d_G(s)}{2}\rceil\}=\{\lfloor \frac{d_{G'}(s)}{2}\rfloor, \lceil \frac{d_{G'}(s)}{2}\rceil\}$. For all $w \in W$, we have $d^+_{\vec{G'}}(w)\leq 1=\ell(w)$ and $d^+_{\vec{G'}}(w) \in \{0,1\}=\{\lfloor \frac{d_G(w)}{2}\rfloor,\lceil \frac{d_G(w)}{2}\rceil\}$. Finally, we have $d^+_{\vec{G'}}(a)=d^+_{\vec{G}}(a)+(8n+k-d^+_{\vec{G}}(a))=8n+k=\frac{d_{G'}(a)}{2}\leq \ell'(a)$. Hence $\vec{G'}$ is best-balanced and $\ell'$-bounded.
\end{proof}
\subsection{Conclusion}\label{conc}
 We here conclude the proof of Theorems \ref{well} and \ref{best}. First observe that both UBWBO and UBBBO are clearly in NP. Next observe that the size of both $(G,\ell)$ and $(G',\ell')$ is polynomial in the size of $(H,k)$. By Lemmas \ref{covtowell} to \ref {convtocov}, we obtain that $(G,\ell)$ is a positive instance of UBWBO if and only if $(H,k)$ is a positive instance of CVC. By Lemmas \ref{covtowell} to \ref{welltobest}, we obtain that $(G',\ell')$ is a positive instance of UBBBO if and only if $(H,k)$ is a positive instance of CVC. As CVC is NP-complete by Theorem \ref{vcdure}, Theorems \ref{well} and \ref{best} follow.

\section*{Statements and declarations}

No funding is available. There are no competing interests. The authors have equally contributed. No data is associated to this manuscript.

\end{document}